\documentclass[english]{amsart}

\usepackage{mathrsfs}
\usepackage{mathtools}
\usepackage{graphicx}
\usepackage{array}
\usepackage{amsmath}
\usepackage{amssymb}
\usepackage{amsthm}
\usepackage{amsfonts}
\usepackage{xcolor}
\usepackage[hyperindex=true,colorlinks=true,linkcolor=blue,citecolor=biblio,urlcolor=biblio]{hyperref}
\usepackage{cleveref}
\crefname{subsection}{subsection}{subsections}

\usepackage{tikz}
\usetikzlibrary{matrix,cd,arrows.meta,fit,calc,positioning}
\tikzcdset{
arrow style=tikz,
diagrams={>={Stealth[scale=0.8]}}
}

\usepackage[shortlabels]{enumitem}
\setlist[enumerate]{label=\rm{(\arabic*)}}

\theoremstyle{plain}
\newtheorem{thm}{Theorem}[section]
\newtheorem{thmIntro}{Theorem}
\newtheorem{corIntro}[thmIntro]{Corollary}
\newtheorem{propIntro}[thmIntro]{Proposition}
\newtheorem*{thm*}{Theorem}
\newtheorem{cor}[thm]{Corollary}
\newtheorem{prop}[thm]{Proposition}
\newtheorem*{prop*}{Proposition}
\newtheorem{lemma}[thm]{Lemma}
\newtheorem{pb}{Problem}

\theoremstyle{definition}
\newtheorem*{Not}{Notation}
\newtheorem{NotNum}[thm]{Notation}
\newtheorem{ex}[thm]{Example}
\newtheorem{defi}[thm]{Definition}
\newtheorem{rem}[thm]{Remark}

\numberwithin{equation}{section}

\newcounter{SExactes}
\newcounter{Diag}
\newcounter{eqIntro}

\DeclareMathOperator{\codim}{codim}

\DeclareMathOperator{\lgth}{length}

\DeclareMathOperator{\Fitt}{Fitt}
\DeclareMathOperator{\Spec}{Spec}

\DeclareMathOperator{\Proj}{Proj}

\DeclareMathOperator{\rk}{rank}
\newcommand{\m}{\mathfrak{m}}

\DeclareMathOperator{\Sy}{S}

\newcommand{\projdC}{\mathbb{P}^2_{\mathbb{C}}}

\newcommand{\projnk}{\mathbb{P}^n}

\newcommand{\projdk}{\mathbb{P}^2}
\newcommand{\coorddk}{\k[x_0,x_1,x_2]}
\newcommand{\coordnk}{\k[x_0,\cdots, x_n]}
\newcommand{\PI}{\mathbb{X}}

\newcommand{\PnX}{\P_X^{n}}

\newcommand{\Pnk}{\P^{n}}

\newcommand{\tX}{\tilde{X}}

\renewcommand{\P}{\mathbb{P}}
\newcommand{\A}{\mathcal{L}}

\renewcommand{\O}{\mathcal{O}}
\renewcommand{\S}{\textnormal{S}}
\newcommand{\V}{\mathbb{V}}
\DeclareMathOperator{\tnV}{V}
\newcommand{\T}{\mathbb{T}}

\newcommand{\I}{\mathcal{I}}
\newcommand{\J}{\mathcal{I}_{\PI}}
\newcommand{\R}{\mathcal{R}}
\newcommand{\E}{\mathcal{E}}
\newcommand{\F}{\mathcal{F}}
\newcommand{\G}{\mathcal{G}}
\newcommand{\bbF}{\mathbb{F}}
\newcommand{\K}{\mathcal{K}}

\newcommand{\p}{p^*}
\renewcommand{\k}{\mathrm{k}}

\renewcommand{\cos}{^\textnormal{c}\textnormal{s}}
\DeclareMathOperator{\tnH}{H}
\newcommand{\Ext}{\mathcal{E}\textnormal{xt}}
\renewcommand{\L}{\mathcal{L}}
\renewcommand{\H}{\mathcal{H}}

\newcommand{\mR}{\mathrm{R}}

\def\restriction#1#2{\mathchoice
              {\setbox1\hbox{${\displaystyle #1}_{\scriptstyle #2}$}
              \restrictionaux{#1}{#2}}
              {\setbox1\hbox{${\textstyle #1}_{\scriptstyle #2}$}
              \restrictionaux{#1}{#2}}
              {\setbox1\hbox{${\scriptstyle #1}_{\scriptscriptstyle #2}$}
              \restrictionaux{#1}{#2}}
              {\setbox1\hbox{${\scriptscriptstyle #1}_{\scriptscriptstyle #2}$}
              \restrictionaux{#1}{#2}}}
\def\restrictionaux#1#2{{#1\,\smash{\vrule height .8\ht1 depth .85\dp1}}_{\,#2}} 

\date{\today}
\title{Torsion of a finite base locus}

\author{R\'emi Bignalet-Cazalet}
\address{Universit\'e de Bourgogne, Institut de Math\'ematiques de Bourgogne,
9 avenue Alain Savary, 
BP 47870 - 21078 Dijon Cedex, France}
\email{remi.bignalet-cazalet@u-bourgogne.fr}



\begin{document}
\definecolor{biblio}{rgb}{0,0.65,1}
\definecolor{xdxdff}{rgb}{0.49,0.49,1}
\definecolor{ttttff}{rgb}{0.2,0.2,1}
\definecolor{zzzzff}{rgb}{0.6,0.6,1}
\definecolor{indigo}{rgb}{0.29,0,0.51}
\definecolor{veronese}{rgb}{0.35,0.4,0.13}

\begin{abstract}
We interpret geometrically the torsion of the symmetric algebra of the ideal sheaf of a zero-dimensional scheme Z defined by $n+1$ equations in an $n$-dimensional variety. This leads us to generalise a formula of A.Dimca and S.Papadima in positive characteristic for a rational transformation with finite base locus. Among other applications, we construct an explicit example of a homaloidal curve of degree $5$ in characteristic $3$, answering negatively a question of A.V.D\'oria, S.H.Hassanzadeh and A.Simis.
\end{abstract}
\maketitle
\section*{Introduction}
The main motivation of this article is the study of rational transformations and homaloidal hypersurfaces over an algebraically closed field $\k$ of any characteristic. Recall that, given a homogeneous square free polynomial $f\in\coordnk$, one defines the \emph{polar map} $\Phi_f:\Pnk\dashrightarrow \Pnk$ by sending $x\in\Pnk$ to $\Big(f_0(x):\ldots:f_n(x)\Big)$ where $f_i=\frac{\partial f}{\partial x_i}$. The hypersurface $F=\lbrace f=0\rbrace \subset \Pnk$ is called \emph{homaloidal} if $\Phi_f$ is birational. It was established by I.V.Dolgachev \cite[Theorem 4]{dolgachev2000polar} that the only homaloidal complex curves are the smooth conics, the unions of three general lines and the unions of a smooth conic with one of its tangent. Furthermore, it was noticed by A.V.D\'oria, S.H.Hassanzadeh and A.Simis \cite{dorHassSim2012polar} that a common property of these complex curves is that the \emph{base locus} of $\Phi_f$, i.e.\ the scheme of zeros of the \emph{jacobian ideal} $I=(f_0,\ldots,f_n)$, is a local complete intersection at each of its points.
\begin{pb}{\cite[Question 2.7]{dorHassSim2012polar}}\label{pbClassif} Let $f\in\coorddk$ be a square free homogeneous polynomial whose polar
map is birational. Is the jacobian ideal locally a
complete intersection at its minimal primes?
\end{pb}
In the spirit of studying the difference between characteristic zero and positive characteristic, we also consider the following \emph{reduction problem}. If $f=q_1^{\alpha_1}\ldots q_m^{\alpha_m}$ is not square free, or equivalently if $F$ is not reduced, the polar map $\Phi_f$ is defined by the mobile part of the linear system generated by $f_0,\ldots,f_n$. Over the field of complex numbers, it was established by A.Dimca and S.Papadima \cite{dimcapap2003hypersurfacecomplements} that $\Phi_f$ birational if and only if so is the polar map $\Phi_{f_{red}}$ associated to $f_{red}=q_1\ldots q_m$. Over a field of positive characteristic, this equivalence trivially fails: in characteristic $2$ for $f= x^2yz$, $\Phi_{f_{red}}$ is birational whereas $\Phi_f$ is not even dominant. This leads to the following problem.
\begin{pb}\label{pbRed}
Over a field of positive characteristic, given $\Phi_f$ dominant, is it birational if and only if so is $\Phi_{f_{red}}$?
\end{pb}
Both problems can be consider from a unified point of view by studying more generally the relation between the \emph{topological degree} $d_t(\Phi)$ of a rational map $\Phi$ and the geometric properties of its base locus. In the polar case, i.e.\ when $\Phi=\Phi_f$ for a homogeneous square free polynomial $f\in\coordnk$ of degree $d$, the base locus $Z=\lbrace f_0=\ldots=f_n=0\rbrace\subset\Pnk$ of $\Phi_f$ coincides with the singular locus of the hypersurface $F=\lbrace f=0\rbrace$. Over $\mathbb{C}$, assuming that this singular locus is finite, the following relation is established by A.Dimca and S.Papadima \cite{dimcapap2003hypersurfacecomplements},
\begin{equation}\label{eqMuIntro}\stepcounter{eqIntro}\tag{\theeqIntro}
d_t(\Phi_f)= (d-1)^n-\mu_f(Z)
\end{equation}
where $\mu_f(Z)$ is the global Milnor number of $F$ (see \cite{Milnor1968SingularPoints} or \Cref{UsMuTau}). Our main goal is to give an algebraic proof and to generalise this formula to the following setting.

Let $X$ be an $n$-dimensional smooth quasi-projective variety over $\k$ and let $\Phi:X\dashrightarrow \Pnk$ be a rational map with zero-dimensional base locus $Z$ determined by a $n+1$-dimensional subspace $\tnV$ of global sections of a line bundle $\mathcal{L}$ over $X$. Our aim is to read off the topological degree $d_t(\Phi)$ of $\Phi$ from properties of the ideal sheaf $\I$ of $Z$, more precisely from the \emph{sheaf of relation} $\E$ defined as the kernel of the canonical evaluation map $ev:\O_X\otimes\tnV\rightarrow \I\otimes\A$.

In \Cref{SectionDT}, we study the projectivization $\pi_1:\PI=\P(\I)\rightarrow X$ of the symmetric algebra of $\I$. We show in particular that it decomposes as the union of the blow-up $\tX$ of $X$ at $\I$ and a \emph{torsion part} $\T_Z$ supported over $Z$. By construction the topological degree of $\Phi$ is equal to that of the restriction to $\tX$ of the lift $\pi_2:\PI\rightarrow \P(\tnV)$ of $\Phi$. In other word $d_t(\Phi)=\deg\Big(c_1\restriction{(\O_{\PI}(1)}{\tX})^n\Big)$.

In this context, we can also consider two other related notions of "naive" topological degrees: the degree $\deg\Big(c_1(\O_{\PI}(1))^n\Big)$ of $\pi_2$ and the algebraic degree of $\Phi$ minus the length of $Z$. In \Cref{CosectionLongueur}~\ref{MuTauL} we show that the second one coincides with the degree of the $0$-cycle $[\V\Big(\vspace{0cm}\cos(\E)\Big)]$ associated to the scheme of zeros of a general cosection $\cos(\E):\E\rightarrow\O_X$ of $\E$. Our main result, proven in \Cref{SectionProof}, asserts in particular that these two naive topological degrees coincide. It also elucidates the relation between these degrees and the topological degree of $\Phi$:

\begin{thmIntro}\label{theorPivot2}
With the notation above, $\PI$ is equidimensional of dimension $n$ and $[\V\Big(\cos(\E)\Big)]=\pi_{1*}c_1\Big(\O_{\PI}(1)\Big)^n$. As a consequence \begin{align*}
d_t(\Phi)=\deg\Big([\V(\cos(\E))]\Big)-\deg\Big(c_1\restriction{(\O_{\PI}(1)}{\T_Z})^n\Big).
\end{align*}

\end{thmIntro}

Let us discuss briefly why this theorem is a generalization of \eqref{eqMuIntro}. This summarizes the content of \Cref{SectionEA}. Recall that another classical invariant of singularities of a hypersurface $F=\lbrace f=0\rbrace$ is the global Tjurina number $\tau_f(Z)$ of $F$ (see \cite{Milnor1968SingularPoints} or \Cref{UsMuTau}). Actually, both Milnor and Tjurina numbers depend on the scheme structure of the singular locus, and, in this sense, they can be defined also for zero-dimensional subscheme $Z$ unrelated to singular hypersurfaces. Having this in mind, we can formulate the following result.
\begin{corIntro}
Formula \eqref{eqMuIntro} holds for any $0$-dimensional subscheme $Z$ defined by $n+1$ global sections of a line bundle $\L$ over a smooth quasi-projective $n$-variety $X$ and over any algebraically closed field.
\end{corIntro}
This corollary follows from the observation that Tjurina numbers compute the degree of $c_1\Big(\O_{\PI}(1)\Big)^n$ whereas Milnor numbers compute the degree of $c_1\Big(\O_{\tX}(1)\Big)^n$. As an immediate application we recover the identity \eqref{eqMuIntro} from the equalities \[\deg\Big([\V(\cos(\E))]\Big)=(d-1)^n-\tau(Z) \hspace{1em}\text{and}\hspace{1em}\deg\Big(c_1\restriction{(\O_{\PI}(1)}{\T_Z})^n\Big)=\mu(Z)-\tau(Z)\] where $\tau(Z)$ and $\mu(Z)$ are the generalised Tjurina and Milnor numbers.

\Cref{ExApplic} presents applications when $X=\projdk$ in which case $\E$ is locally free of rank $2$. The first application is motivated by the following situation. Tjurina numbers appear in a natural way in the classification of complex \emph{free curves}. These are the plane curves $F=\lbrace f=0\rbrace$ of degree $d$ whose jacobian ideal sheaves $\I$ have a locally free resolution of the form.
\begin{equation*}
\begin{tikzcd}[row sep=3em,column sep=0.5cm,minimum width=2em]
  0 \ar{r}& \O_{\P^2}(3-2d)\oplus \O_{\P^2}(-d)\ar{r}& \O_{\P^2}(1-d)^{3} \ar{r}&  \I \ar{r}&0.
\end{tikzcd}
\end{equation*}
It was established by A.A.du Plessis and C.T.C.Wall in \cite{duplessisWall19991higlysingular} that these curves are characterized by the following identity:
\begin{equation}\label{eqTauIntro}\stepcounter{eqIntro}\tag{\theeqIntro}
d-2=(d-1)^2-\tau_f(Z).\end{equation}

A first application is a generalisation of the numerical characterization \eqref{eqTauIntro} in arbitrary characteristic to locally free sheaves of rank $2$:
\begin{thmIntro}
Let $\E$ be the sheaf of relation of an ideal sheaf generated by three homogeneous polynomials of degree $d-1\geq 0$ in $3$ variables. Then \begin{equation*}\label{eqTauLibre}d-2 \leq (d-1)^2-\tau(Z)\end{equation*} and equality occurs if and only if $\E\simeq \O_{\projdk}(-1)\oplus \O_{\projdk}(2-d)$.
\end{thmIntro}

As a second application in \Cref{ExApplic}, we answer negatively \Cref{pbClassif}:

\begin{propIntro}
The curve $F=\V\Big((x_1^2+x_0x_2)x_0(x_1^2+x_0x_2+x_0^2)\Big)$ is homaloidal if and only if the base field $\k$ has characteristic $3$.
\end{propIntro}

We also answer negatively \Cref{pbRed} by producing an explicit homaloidal curve in characteristic $101$ whose polar has topological degree $3$ whereas the polar of its reduction has topological degree $5$.

\

The explicit computations given in this paper were made using basic functions of Macaulay2 and the Cremona package also running on Macaulay2 \cite{stagliano2017Mac2Pack}. The corresponding codes are available on request.

\section{Topological degrees via the symmetric algebra}\label{SectionDT}
We first recall some facts about the symmetric and the Rees algebras (or blow-up algebra) of an ideal before giving the definition of \emph{topological degree} and \emph{naive topological degrees}.

\subsection{Rees and symmetric algebras}\label{subSecReesSym}
Given an ideal sheaf $\I$ on a smooth variety $X$ of dimension $n$, recall that the blow-up $\tX$ of $X$ at $\I$ is the Proj of the \emph{Rees algebra}  \[R(\I)=\O_X\oplus \I t\oplus \I^2t^2\oplus \cdots =\underset{i=0}{\overset{\infty}{\oplus}}\I^it^i\subset \O_X[t]\] of $\I$. Denoting $Z=\V(\I)$, we say also that $\tX$ is the blow-up of $X$ along $Z$. We denote by $\textnormal{S}(\I)=\oplus_{i\geq 0}\textnormal{S}^i(\I)$ the \emph{symmetric algebra} of $\I$ and by $\PI$ the projectivization $\P(\I)=\Proj(\Sy(\I))$ of $\I$ with its bundle map $\pi_1:\PI\rightarrow X$.

The natural surjection $q:\S(\I)\rightarrow \R(\I)$ defines a closed embbeding of $\tilde{X}$ in $\PI$. When $q$ is an isomorphism $\I$ is said of \emph{linear type} \cite{Vasconcelos2005Int}. This is the case for instance when $\I$ is locally generated by a regular sequence \cite[Example 1.2]{Vasconcelos2005Int}.

Otherwise the images by $\pi_1$ of the irreducible components of $\PI$ different from $\tX$ are contained in the support of $Z$. Indeed, over the set $U=X\backslash Z$, we have $\I_U=\O_U$, so that $\restriction{\tX}{U}=\restriction{\PI}{U}=\pi_1^{-1}(U)$. This justifies the following definition:

\begin{defi}\label{defCompTorsion}
An irreducible component of $\PI$ different from $\tX$ is called a \emph{torsion component} of $X$. The union of the torsion components is called the \emph{torsion part} of $\PI$, denoted by $\T_Z$.
\end{defi}
The following lemma provides a description of the torsion components supported over the generic points of the irreducible components of $Z$. Namely, letting $Z_i$ be an irreducible component of \( Z_{red} \) we consider $A=\O_{X,Z_i}$ and $I$ the image of $\I$ in $A$.

\begin{lemma}\label{lemmaLCI}
Let $X=\Spec(A)$ be the spectrum of a regular local ring essentially of finite type with maximal ideal $\m$ and residue field $\kappa(\m)$. Let $I\subset \m$ be an $\m$-primary ideal minimally generated by $r+1$ elements which do not form a regular sequence. Then $\PI$ is the union of $\tX$ and a unique other irreducible component contained in $\pi_1^{-1}(\m)$ whose reduction is isomorphic to $\P_{\kappa(\m)}^r$.
\end{lemma}

\begin{proof}
Let:
\begin{center}
\begin{tikzpicture}
  \matrix (m) [row sep=0.5em,column sep=2em,minimum width=2em]
  {
     \node(c){$A^m$}; & \node(a){$A^{r+1}$}; &\node(){}; &\node(d){$I$}; & \node(e){$0$}; \\};

  \path[-stealth]
    (c) edge node[above]{$M$} (a)
    (a) edge node[above]{$(\phi_0\;\ldots\;\phi_r)$} (d)
    (d) edge (e);
\end{tikzpicture}
\end{center}
be a minimal presentation of $I$. Then $\PI=\P(I)$ is isomorphic to the closed subscheme of $\P_A^r$ defined by the entries of the row matrix $(y_0\;\ldots\;y_n)\cdot M$ \cite[A.III.69.4]{Bourbaki2007Algebre}. Since $A$ is local and $\lbrace\phi_0,\ldots,\phi_r\rbrace$ is a minimal set of generators of $I$, all the entries of $M$ are elements of $\m$ \cite[19.4]{eisenbud1995algebra}. It follows that $\pi_1^{-1}(\lbrace \m \rbrace )=\P^r_{\m}\simeq \P_{\kappa(\m)}^r$. The exceptional divisor in $\tX$ above the point $\m$ is then the intersection of $\PI$ and $\tX$.
\end{proof}

From a practical point of view, it might be difficult to determine how the torsion components vary from a given presentation, as illustrated by the following example:
\begin{ex}
Consider the ideal $I$ of $A=\k[x,y,z]$ given by the following resolution:
\begin{small}
\begin{center}
\begin{tikzpicture}
  \matrix (m) [row sep=0.5em,column sep=0.5cm,minimum width=2em]
  {
     \node(b){$0$}; &\node(c){$A^3$};&\node(){}; &\node(){};  & \node(a){$A^4$}; &\node(){}; &\node(d){$I$}; & \node(e){$0$}; \\};

  \path[-stealth]
  (b) edge (c)
    (c) edge node[above]{\begin{scriptsize}
    $\begin{pmatrix}
    0  &  xz & y^2 \\
    0  &  x & xy \\
   x &  y & y      \\
 y &  z  & x      \\    \end{pmatrix}$
    \end{scriptsize}  }(a)
    (a) edge node[above]{$(\phi_0\;\ldots\; \phi_3)$} (d)
    (d) edge (e);
\end{tikzpicture}
\end{center}
\end{small}
Above the line $\lbrace x=y=0\rbrace$ in $X=\Spec(A)$, the torsion component is $\lbrace x=y=0\rbrace\times \P_\k^2$ but above the point $\lbrace x=y=z=0\rbrace$, the torsion component is $\lbrace x=y=z=0\rbrace\times \P_\k^3$.
\end{ex}

This motivates why, in the following, we will assume that $\V(\I)$ is zero dimensional.

In the case when $\I\otimes \A$ is generated by $n+1$ sections for some line bundle $\A$ over $X$, let \begin{center}
\begin{tikzpicture}
  \matrix (m) [row sep=0.5em,column sep=2em,minimum width=2em]
  {
     \node(c){$\F$}; &\node(a){$\O_X^{n+1}$}; &\node(){}; &\node(d){$\I\otimes\A$}; &\node(e){$0$}; \\};

  \path[-stealth]
    (c) edge node[above]{$s$} (a)
    (a) edge node[above]{$(\phi_0\;\ldots\; \phi_n)$} (d)
    (d) edge (e);
\end{tikzpicture}
\end{center}
be a locally free presentation of $\I\otimes\A$. The map $s$ can be interpreted dually as the data of $n+1$ sections of $\F^\vee$. Recall that $\Fitt_{n}\I$ is the ideal sheaf generated by the common vanishing of these $n+1$ sections and $\V(\Fitt_{n}\I)\subset Z$ \cite[20]{eisenbud1995algebra}. With this definition, we have:

\begin{cor}\label{LCI}
Let $X$ be a smooth variety of dimension $n$ and let $\I$ be an ideal sheaf over $X$. Denoting $Z=\V(\I)$, assume that $\codim(Z)=n$ and that $\I\otimes \A$ is generated by $n+1$ sections for some line bundle $\A$ over $X$. Then the images of the torsion components of $\PI$ in $X$ are precisely the points of the subscheme $\V(\Fitt_{n}\I)$. Moreover, each torsion component with its reduced structure is isomorphic to $\projnk$.
\end{cor}

\begin{proof}
Since $Z$ is zero-dimensional, any $z\in Z$ is in an affine open set $U=\text{Spec}(A)$ of $X$ over which $\A$ is trivial. So over $U$, let
\begin{center}
\begin{tikzpicture}
  \matrix (m) [row sep=0.5em,column sep=2em,minimum width=2em]
  {
     \node(c){$\O_U^m$}; &\node(a){$\O_U^{n+1}$}; &\node(d){$\restriction{\I}{U}$}; &\node(e){$0$}; \\};

  \path[-stealth]
    (c) edge node[above]{$M$} (a)
    (a) edge (d)
    (d) edge (e);
\end{tikzpicture}
\end{center}
be a presentation of $\I$.

By \cite[Proposition 20.6]{eisenbud1995algebra}, the scheme $\V(\Fitt_{n}\I)$ is the subscheme of $Z$ consisting of points $z\in \projnk$ at which $\I_z$ can not be generated by $n$ elements.  By \Cref{lemmaLCI}, only two situations can occur. Either $Z$ is a local complete intersection at $z$ i.e.\ $\I_z$ is generated by a local regular sequence. Hence the Rees algebra and the symmetric algebra coincide locally as explained before \Cref{defCompTorsion}.

Or $Z$ is not a local complete intersection at $z$ and then, since $\codim\Big(\V(\I)\Big)=n$, $\V(\Fitt_{n}\I)$ is exactly the support where $(\phi_0,\ldots ,\phi_n)$ can not be a local regular sequence. But the ideal of $\PI$ in $\P^n_U$ is $(y_0\hspace{0.1cm}\ldots\hspace{0.1cm} y_n)\cdot M$ where $(y_0,\ldots,y_n)$ are the coordinates of the second factor so the scheme $\V(\Fitt_{n}\I)$ is exactly the scheme of point of $z\in Z$ such that set-theoretically $\pi_1^{-1}(\lbrace z\rbrace)=\P^n_z$.

\end{proof}

\begin{NotNum}\label{NotationTorsion}
For every $z\in Z$, we let $\T_z$ be the scheme-theoretic fibre of the restriction of $\pi_1$ to $\T_Z$. By \Cref{LCI}, $\T_z$ is set-theoretically equal to $\P^n_z$ so $T_z=[\T_z]\cdot c_1\Big(\O_{\PI}(1)\Big)^n$ is a $0$-cycle on $\PI$. We denote by $T_Z$ the $0$-cycle $\underset{z\in Z}{\sum} T_z$.
\end{NotNum}

\subsection{Geometric interpretation of the torsion}
From now on, our setting is as follows: $X$ is a smooth $n$-variety over $\k$, $\mathcal{L}$ is a line bundle over $X$ and $\tnV$ is an $(n+1)$-dimensional subspace of $\text{H}^0(X,\L)$. We denote $\Phi:X\dashrightarrow \P(\tnV)\simeq \Pnk$ the associated rational map. Recall that the base ideal sheaf $\I$ of $\Phi$ is the image of the evaluation map $ev:\tnV\otimes \L^\vee\rightarrow \O_X$. By the universal property of blow-up, $\tX$ is isomorphic to the graph $\Gamma$ of $\Phi$, that is, the closure in $X\times \Pnk$ of the graph of the restriction of $\Phi$ to its domain of definition.

Let $p_1: \PnX=\P(\tnV\otimes \O_X)\rightarrow X$ be the structure map and let
\begin{center}
\begin{tikzpicture}
  \matrix (m) [row sep=0.5em,column sep=2em,minimum width=2em]
  {
     \node(c){$\F$}; &\node(a){$\O_X^{n+1}$}; &\node(d){$\I\otimes\A$}; &\node(e){$0$}; \\};

  \path[-stealth]
    (c) edge node[above]{$s$} (a)
    (a) edge node[above]{$ev$} (d)
    (d) edge (e);
\end{tikzpicture}
\end{center}
be a locally free presentation of $\I\otimes \A$. The map $ev$ determines a closed embedding $\P(\I\otimes \A)\hookrightarrow \PnX$ as the zero scheme of the global section $\sigma\in \tnH^0(\Pnk,\O_{\PnX}(1)\otimes \p_1\F^\vee)$ deduced from the composition of $\p_1 s$ with the canonical surjection $\tnV\otimes \O_{\PnX}\rightarrow \O_{\PnX}(1)$. Since $\P(\I\otimes \A)\simeq\P(\I)$ \cite[II.7.9]{hartshorne1977algebraic}, this provides a closed embedding $\PI\hookrightarrow \PnX$.

Summing up, we have the following commutative diagram \eqref{Diag1}:
\begin{equation}\label{Diag1}\stepcounter{Diag}\tag{D\theDiag}
\begin{tikzcd}[row sep=0.73cm,column sep=1.8cm,minimum width=2em]
       &   \PnX \ar[dddl,"p_1" description]\ar[dddr,"p_2"description]&       \\
       &   \PI  \ar[u,hook,"\iota"description]\ar[ddl,"\pi_1"description]\ar[ddr,"\pi_2"description]& \\
       &   \tX   \ar[u,hook]\ar[dl,"\sigma_1"description]\ar[dr,"\sigma_2"description]&  \\
 X \ar[rr, dashed, "\Phi"]&          &        \Pnk       
\end{tikzcd}
\end{equation}

Recall that the \emph{topological degree} of a dominant rational map $\Phi:X\dashrightarrow Y$ between irreducible varieties $X$ and $Y$ of the same dimension is defined as the degree $d_t(\Phi)=[\textnormal{Frac}(Y):\textnormal{Frac}(X)]$ of the induced extension between their respective fields of rational functions. In our setting, $Y=\P^n$ and $d_t(\Phi)$ can be interpreted alternatively as the degree of the $0$-cycle $c_1\Bigl(\restriction{\O_{\PnX}(1)}{\tX}\Bigr)^n$ on $\tX$. Since $\sigma_1$ is birational we have thus:
\[d_t(\Phi)=\deg\Big(c_1(\restriction{\O_{\PnX}(1)}{\tX})^n\Big)=\deg\Big(\sigma_{1*}(c_1(\restriction{\O_{\PnX}(1)}{\tX})^n)\Big).\]

By \Cref{LCI}, $c_1(\restriction{\O_{\PnX}(1)}{\PI})^n$ is also a $0$-cycle on $\PI$ so we can set the following definition.
 
\begin{defi}\label{VraiDegTopNaif}
With the notation in \eqref{Diag1} the degree of $c_1(\restriction{\O_{\PnX}(1)}{\PI})^n$ is called the \emph{first naive topological degree} of $\Phi$.
\end{defi}

Intuitively, the difference between the first naive topological degree and the actual topological degree reflects a difference between the symmetric algebra and the Rees algebra, see \Cref{theorMuTau} below for a precise statement.

Now, let $\E$ be the kernel of the evaluation map $ev:\O_{X}^{n+1}\rightarrow\I\otimes\A$ and let $\alpha:\O_X^{n+1}\rightarrow \O_X$ be a generic map. Since $\E$ has rank $n$, the zero locus $\V(\cos_\alpha)$ of the composition $\cos_{\alpha}=\alpha\circ \gamma$ is a $0$-dimensional subscheme  of $X$. \begin{center}
\begin{tikzpicture}
  \matrix (m) [row sep=1em,column sep=2em,minimum width=2em]
  {
    \node(a){$0$}; &\node(b){$\E$}; &\node(c){$\O_X^{n+1}$}; & \node(d){$\I\otimes \A$}; & \node(e){$0$}; \\
    \node(){}; &\node(){}; &\node(z){$\O_X$}; & \node(){}; & \node(){}; \\};
  \path[-stealth]
    (a) edge (b)
    (b) edge node [above]{$\gamma$} (c)
    (c) edge node [above] {$ev $} (d)
    (d) edge (e)
    (c) edge node[right]{$\alpha$}(z)
    (b) edge node[below]{$\cos_{\alpha}$}(z);
\end{tikzpicture}
\end{center}

In the proof of \Cref{theorPivot2}, we will establish in particular that the cycle class $[\V(\cos_\alpha)]$ of $\V(\cos_\alpha)$ is independent on the choice of a generic map $\alpha$ so, anticipating, we set the following definition.

\begin{defi}\label{DegTopNaif}
The \emph{second naive topological degree} of $\Phi$ is the degree of the $0$-cycle $[\V\Big(\hspace{0cm}\cos(\E)\Big)]$ of a generic cosection $\cos(\E)$ of $\E$.
\end{defi}

\begin{rem} If $\E$ is locally free, $[\V\Big(\hspace{0cm}\cos(\E)\Big)]$ simply coincides with the top Chern class $c_n(\E^\vee)$ of $\E^\vee$. This is no longer true when $\E$ is not locally free. For instance the sheaf $\E$ of relations of the ideal sheaf $\I=(x_1^2-x_1x_3 , x_2^2-x_2x_3 , x_1x_2 , x_0x_3 )$ of $\P^3$ satisfies $c_3(\E^\vee)=4$ whereas $\deg\Big( [\V(\cos(\E))]\Big)=2$ as we can check from the resolution of $\E$:
\begin{center}
\begin{tikzpicture}
  \matrix (m) [row sep=1cm,column sep=0.75cm,minimum width=2em]
  {
      \node(a){$0$}; &\node(c){$\O_{\P^3}(-3)^2$}; &\node(d){$\O_{\P^3}(-1)^2\oplus\O_{\P^3}(-2)^3$}; &\node(e){$\E$}; & \node(f){$0$.}; \\
	};
  \path[-stealth]
    (a) edge (c)
    (c) edge (d)
    (d) edge (e)
    (e) edge (f);
\end{tikzpicture}
\end{center}

\end{rem}
\section{Proof of Theorem \ref{theorPivot2}}\label{SectionProof}
Recall the settings of \Cref{theorPivot2}, we assume that $n\geq 2$, $\codim(Z)=n$ and that the map $\Phi$ is dominant.

By definition, the first naive topological degree is the length of the $0$-scheme $W$ of a general section of $\O_{\PI}(1)^n$. Our strategy to show \Cref{theorPivot2} is now to push forward the following exact sequence:
\begin{equation}\label{suiteFond}\stepcounter{SExactes}\tag{E\theSExactes}
\begin{tikzcd}[row sep=1cm,column sep=0.75cm,minimum width=2em]
0 \ar{r}& \K \ar{r}& \O_{\PI}^n \ar{r}& \O_{\PI}(1) \ar{r}& \O_W(1) \ar{r}& 0
\end{tikzcd}
\end{equation}
where $\K$ is by definition the kernel of the map $\O_{\PI}^n\rightarrow\O_{\PI}(1)$.
So, applying $\pi_{1*}$ to \eqref{suiteFond} and assuming that $\mR^1\pi_{1*}\Bigl(\mathcal{K}\Bigr)=\mR^1\pi_{1*}\Bigl(\I_W(1)\Bigr)=0$, we have
\begin{center}
\begin{tikzpicture}
  \matrix (m) [row sep=3em,column sep=2.5em,minimum width=2em]
  {
	\node(c){$\O_X^n$}; &\node(d){$\pi_{1*}\O_{\PI}(1)$}; &\node(e){$\pi_{1*}\O_W(1)$}; & \node(f){$0$.}; \\};
  \path[-stealth]
    (c) edge (d)
    (d) edge (e)
    (e) edge (f);
\end{tikzpicture}
\end{center}
We emphasize that $\I$ is not locally free so $\pi_{1*}(\O_{\PI}(1))$ might a priori be different from $\I$ (see Stack project, 26.21. Projective bundles, \href{https://stacks.math.columbia.edu/tag/01OA}{example 26.21.2}). However our strategy is to prove that these coincide in this case.

We use the same notation for the sheaves and their push forward by $\PI\overset{\iota}{\hookrightarrow}\PnX$. Thus, the strategy is to provide that $\mR^1p_{1*}\Bigl(\mathcal{K}\Bigr)=\mR^1p_{1*}\Bigl(\I_W(1)\Bigr)=0$ and $p_{1*}\Bigl(\O_{\PI}(1)\Bigr)=\I\otimes \A$ in order to get the sequence:
\begin{equation}\label{exSeq5}\stepcounter{SExactes}\tag{E\theSExactes}
\begin{tikzcd}
\O_X^n \ar{r}&  \I\otimes \A  \ar{r}& p_{1*}\O_W(1)  \ar{r}& 0.
\end{tikzcd}
\end{equation}
As we will explain below, $[p_{1*}\O_W(1)]$ will turn out to be precisely the cycle $[\V(\cos_\alpha)]$ which by definition verifies the following exact sequence:
\begin{center}
\begin{tikzpicture}
  \matrix (m) [row sep=3em,column sep=2em,minimum width=2em]
  {
    \node(b){$\mathcal{E}$}; &\node(c){$\O_X$}; &\node(d){$\O_{ \V(\cos_{\alpha})}$}; & \node(e){$0$.}; \\};
  \path[-stealth]
    (b) edge node[above]{$\cos_{\alpha}$} (c)
    (c) edge (d)
    (d) edge (e);
\end{tikzpicture}
\end{center}
This will show eventually \Cref{theorPivot2}.

\subsection{Cohomological preliminaries}

\begin{lemma}\label{Vanishing}\label{thmSubLin}
The following vanishings hold:
\begin{enumerate}[label=\rm{\it(\roman*)}]
\item\label{Vanishing1} $\mR^1p_{1*}\I_{\PI}(1)=0$,
\item\label{Vanishing2} $\mR^{i+1}p_{1*}\O_{\PI}(-i)=0$ for every $i\in\lbrace 0,\ldots, n-1\rbrace$,
\item\label{Vanishing3} $\mR^{i}p_{1*}\O_{\PI}(-i)=0$ for every $i\in\lbrace 1,\ldots, n-1\rbrace$.
\end{enumerate}
\end{lemma}

\begin{proof}
Under the assumption that $\dim(Z)=0$, by \cite[Corollary 2.9]{Big2018ResSymALg}, the ideal $\I_{\PI}$ has a locally free resolution of the following form:
\begin{equation*}\label{resSym}\tag{\text{$\G_{\bullet}$}}
\begin{tikzcd}[row sep=3em,column sep=1em,minimum width=2em]
0 \ar{r}& \G_{n+1} \ar{r}& \G_n \ar{r}& \ldots \ar{r}& \G_2 \ar{r}& \G_1 \ar{r}& \J \ar{r}& 0
\end{tikzcd}
\end{equation*}
where $\G_i=\underset{j=1}{\overset{i}{\oplus}}\p\G_{ij}\otimes \O_{\PnX}(-j)$ when $i\in\lbrace 1,\ldots,n\rbrace$ and $\G_{n+1}= \p\G'_{n+1}\otimes \O_{\PnX}(-1)$ for some locally free sheaves $\G_{ij}$ and $\G'_{n+1}$ over $X$.

Now, a diagram chasing in \eqref{resSym} shows that $\mR^1p_{1*}\J(1)=0$ provided that $\mR^{k}p_{1*}\Big(\G_k(1)\Big)=0$ for all $ k\in\lbrace 1,\ldots,n+1\rbrace$. By Kunneth formula, those vanishings are verified if:
\begin{itemize}
\item $\tnH^{k}\Bigl(\projnk,\O_{\projnk}(-j+1)\Bigr)=0$ for all $k\in\lbrace 1,\ldots,n\rbrace$ and all $j\in\lbrace 1,\ldots,k\rbrace$,
\item $\tnH^{n+1}\Bigl(\projnk,\O_{\projnk}(-2)\Bigr)=0$
\end{itemize}

The only non trivial case to check is when $k=n$. But:
\begin{center}
$\tnH^{n}\Bigl(\projnk,\O_{\projnk}(-j+1)\Bigr)\simeq \tnH^{0}\Bigl(\projnk,\O_{\projnk}(j-n-2)\Bigr)^\vee=0$
\end{center}
because $j\leq n$.

For \ref{Vanishing2} and \ref{Vanishing3}, since $\O_{\PI}=\O_{\PnX}/\I_{\PI}$, the assertions follow from the same argument after twisting the complex \eqref{resSym} by $\O_{\PnX}(-i)$ for every $i\in\lbrace 0,\ldots, n-1\rbrace$.
\end{proof}

\begin{lemma}\label{imageDirecte3}
We have $p_{1*}\Bigl(\O_{\PI}(1)\Bigr)=\I\otimes \A$.
\end{lemma}

\begin{proof}
First, $\O_{\PnX}(1)$ being the relative ample line bundle of the projective bundle $\PnX=\P\Big(\O_X^{n+1}\Big)$, we have $p_{1*}\O_{\PnX}(1)=\O_X^{n+1}$.

Moreover, since $\I_{\PI}(1)$ is the image of the canonical map $\p_{1}\E\rightarrow \O_{\PnX}(1)$, we let $\H$ be the kernel of this surjection and we write the exact sequence:
\begin{equation*}
\begin{tikzcd}[row sep=3em,column sep=1em,minimum width=2em]
0 \ar{r}& \H \ar{r}& \p_1 \E \ar{r}& \I_{\PI}(1) \ar{r}& 0.
\end{tikzcd}
\end{equation*}

Since $p_{1*}\p_1\E\simeq \E$ and $\mR^1p_{1*}\p_1\E=0$, applying $p_{1*}$ to this exact sequence, we get:
\begin{equation}\label{pfX1}\tag{a}
\begin{tikzcd}[row sep=3em,column sep=1em,minimum width=2em]
0 \ar{r}& p_{1*}\H \ar{r}& \E \ar{r}& p_{1*}\I_{\PI}(1) \ar{r}& \mR^1 p_{1*}\H \ar{r}& 0.
\end{tikzcd}
\end{equation}

Also, since we proved that $\mR^1p_{1*}\I_{\PI}(1)=0$, applying $p_{1*}$ to the canonical exact sequence
\begin{equation*}
\begin{tikzcd}[row sep=3em,column sep=1em,minimum width=2em]
0 \ar{r}& \I_{\PI}(1) \ar{r}& \O_{\PnX}(1) \ar{r}& \O_{\PI}(1) \ar{r}& 0
\end{tikzcd}
\end{equation*}
we get
\begin{equation*}\label{pfX2}\tag{b}
\begin{tikzcd}[row sep=3em,column sep=1em,minimum width=2em]
0 \ar{r}& p_{1*}\I_{\PI}(1) \ar{r}& \O_X^{n+1} \ar{r}& p_{1*}\O_{\PI}(1) \ar{r}& 0.
\end{tikzcd}
\end{equation*}

The exact sequences \eqref{pfX1} and \eqref{pfX2} fit into the following commutative diagram:
\begin{equation*}
\begin{tikzcd}[row sep=0.8em,column sep=1em,minimum width=2em]
 & 0 \ar{d}&  &  & & \\
 & p_{1*}\H \ar{d}& & 0 \ar{d}& & \\
& \E \ar{d}\arrow[rr, "\simeq"]&  & \E \ar{d}& &\\
0 \ar{r}& p_{*}\I_{\PI}(1) \ar{rr}\ar{d}& & \O_X^{n+1} \ar{r}\ar{d}& p_{1*}\O_{\PI}(1) \ar{r}\arrow[d,phantom,"{\rotatebox{90}{=}}"]& 0  \\
0 \ar{r} &\mR^1p_{1*}\H \ar{rr}\ar{d}&  & \I_Z\otimes\L \ar{r}\ar{d}& p_{1*}\O_{\PI}(1) \ar{r}& 0 \\
 & 0 &  & 0 &   &  \\
\end{tikzcd}
\end{equation*}
where \eqref{pfX1} is the left column, \eqref{pfX2} is the central row and the map $\I_Z\rightarrow p_{1*}\O_{\PI}(1)$ in the bottom row is the canonical morphism associated to the projectivization of $\I_Z$. This morphism is an isomorphism at $X\backslash Z$ and therefore $\I_Z\otimes \L\rightarrow p_{1*}\O_{\PI}(1)$ is injective because $\I_Z$ is torsion free. Hence $p_{1*}\H\simeq 0\simeq \mR^1p_{1*}\H$ and $p_{1*}\O_{\PI}(1)\simeq \I_Z\otimes\L$.

\end{proof}

\subsection{Degree of cycles}
As above, let $W\subset \PI$ be the intersection of $\PI$ with $n$ general relative hyperplanes of $\PnX$ so that $[W]=c_1\Big(\O_{\PI}(1)\Big)^n$.

\begin{proof}[Proof of Theorem \ref{theorPivot2}]\label{proofPivot}

Consider the following exact sequence:
\begin{equation}\label{Koszul}\tag{Kz}
\begin{tikzcd}[row sep=0.5em,column sep=1em,minimum width=2em]
0 \ar{r}& \K \ar{r}& \O_{\PI}^n \ar{dr}\arrow[rr, "{\beta'}"]& & \O_{\PI}(1) \ar{r}& \O_W(1) \ar{r}& 0 \\
    &  &  & \I_W(1) \ar{ur}\ar{dr}& & & \\     
    &  & 0 \ar{ur}& & 0 & & \\
\end{tikzcd}
\end{equation}

We claim that $$\mR^1p_{1*}\Bigl(\I_W(1)\Bigr)=\mR^1p_{1*}\Bigl(\mathcal{K}\Bigr)=0.$$

Indeed by \Cref{LCI}, $\PI$ decomposes as the union of $\tX$, the blow-up of $X$ at $\I$, and the torsion part $\T_Z$, possibly empty, whose reduced structure is $\P_{Z'}^n$ for a set $Z'\subset Z$.

So the Koszul complex provides a resolution 
\begin{equation*}
\begin{tikzcd}[column sep=0.5cm,minimum width=2em]
0 \ar{r}& \O_{\PI}(-n+1) \ar{r}& \ldots \ar{r}& \O_{\PI}(-1)^{\binom{n}{2}} \ar{r}& \O_{\PI}^n \ar{r}& \I_W(1) \ar{r}& 0
\end{tikzcd}
\end{equation*}
of $\I_W(1)$ and the desired vanishings follow from \Cref{Vanishing}~\ref{Vanishing2}.

Since $p_{1*}\O_{\PI}^n\simeq\O_X^n$, $p_{1*}\O_{\PI}(1)\simeq\I\otimes\A$, $\mR^1p_{1*}\Bigl(\I_W(1)\Bigr)=0$ and $\mR^1p_{1*}(\mathcal{K})=0$, pushing forward by $p_1$ the exact sequence \eqref{Koszul}, we obtain the following commutative diagram:

\begin{equation}\stepcounter{Diag}\tag{D\theDiag}\label{serpent}
\begin{tikzcd}[row sep=1em,column sep=1.5em,minimum width=2em]
 &    &  0 \ar{d}& & p_{1*}\K \ar{d}&  \\
 &    & \O_X^n \arrow[d, "{\beta}"]\arrow[rr,phantom, "="]& &\O_X^n \arrow[d, "{p_{1*}\beta'}"]& \\
0\ar{r}& \E  \ar{r}\arrow[d,phantom,"\rotatebox{90}{=}"]& \O_X^{n+1} \ar[rr,"(\phi_0\;\ldots\;\phi_n)"]\arrow[d,"\alpha"] & &\I\otimes\A \ar{d}\ar{r}& 0 \\
 & \E \ar[r,"{\cos_{\alpha}}"]& \O_X \ar{d}\ar{rr}& & p_{1*}\O_W(1)\ar{r}\ar{d}& 0 \\
 & &0 & & 0 & \\
\end{tikzcd}
\end{equation}
where $\alpha$ is the cokernel map of the vertical map $\beta:\O_X^n\rightarrow \O_X^{n+1}$. Hence $p_{1*}\K=\ker(\cos_{\alpha})$ and \eqref{serpent} implies that 
\begin{align*}
p_{1*}(\O_W)\simeq\O_{\V(\cos_{\alpha})}
\end{align*}
as in \Cref{DegTopNaif}. So in the end, we have that $[\V(\cos_{\alpha})]=[p_{1*}W]$. Since all the generic map $\alpha$ as in \eqref{exSeq5} can be obtained as cokernel of a generic map $\beta:\O_X^n\rightarrow \O_X^{n+1}$, $[\V(\cos_\alpha)]$ does not depend on the generic map $\alpha$ so that we can write $[\V\Big(\hspace{0cm}\cos(\E)\Big)]$ for a generic cosection $\cos(\E)$.

The fact that $\deg(W)=\deg(p_{1*}W)$ comes from the decomposition of $W$. Indeed, $\PI$ decomposes into the graph $\tX$ and the torsion part $\T_Z$ supported on $\P^n_{\Fitt_n(Z)}$. Hence, we have the equality 
\begin{center}
$[W]=[\PI]\cdot c_1\Big(\O_{\PnX}(1)\Big)^n=[\tX]\cdot c_1\Big(\O_{\PnX}(1)\Big)^n+[\T_Z]\cdot c_1\Big(\O_{\PnX}(1)\Big)^n$.
\end{center}

Since $\tX$ is irreducible and $\sigma_1:\tX\rightarrow X$ birational, we have \[\deg\Big([\tX]\cdot c_1\Big(\O_{\PnX}(1)\Big)^n\Big)=\deg\Bigl(\sigma_{1*}([\tX]\cdot c_1\Big(\O_{\PnX}(1)\Big)^n)\Bigr).\] Moreover, as a consequence of \Cref{theorPivot2} we have:
\begin{center}
$d_t(\Phi)=\deg\Big([\V(\cos(\E)] \Big)-\deg\Big(p_{1*}([\T_Z]\cdot c_1\Big(\O_{\PnX}(1)\Big)^n)\Big)$.
\end{center}
\end{proof}

\section{Measure of the difference between Rees and symmetric algebras}\label{SectionEA}

We relate now the topological degree and the naive topological degree with the notions of Milnor and Tjurina numbers. For the rest of this section, $\I$ is the ideal of a rational map $\Phi=(\phi_0:\ldots:\phi_n)$ associated to an $n+1$-subspace $\tnV$ of $\tnH^0(X,\L)$ where $\L$ is a line bundle over $X$. We denote by $Z$ the base scheme $\V(\I)$ in $X$ and we assume that $\dim(Z)=0$.

\subsection{Generalized Milnor and Tjurina numbers}
\begin{Not}
We set temporarily as a notation that $\delta^n=\deg\Big(c_1(\L)^n\Big)$ which as to be understood as $\delta=\deg\Big(c_1(\L)\Big)$ when $X$ is the projective space $\P^n$.
\end{Not}

\begin{defi}\label{defTjurina} \label{defMilnor}
With notation as in \Cref{NotationTorsion}, for every $z\in Z$, put:

\begin{itemize}
\item $\tau(Z,z)=\lgth(\O_{Z,z})$

\item $\mu(Z,z)=\begin{cases}\tau(Z,z)\hspace{0.3cm}\text{if }Z\text{ is a local complete intersection at }$z$,\\
\tau(Z,z)+\deg(T_z)\hspace{0.3cm}\text{otherwise.}
\end{cases}$
\end{itemize}

We let $\tau(Z)=\underset{z\in Z}{\sum}\tau(Z,z)$ and $\mu(Z)=\underset{z\in Z}{\sum}\mu(Z,z)$.
 
\end{defi}

As a direct application of \Cref{theorPivot2}, we obtain:

\begin{prop}\label{CosectionLongueur} \label{theorMuTau}
The following equalities hold:
\begin{enumerate}[label=\rm{\it(\roman*)}]
\item \label{MuTauL} $\deg\Big( [\V(\cos(\E))]\Big) = \delta^n-\tau(Z)$
\item \label{MuTauT} $d_t(\Phi) = \delta^n-\mu(Z) $
\item \label{MuTauDiff} $d_t(\Phi)- \deg( \cos(\E))=\mu(Z)-\tau(Z)= \deg(T)=\deg(p_{1*}T)$.
\end{enumerate}
\end{prop}

\begin{proof}
Looking back at the diagram \eqref{serpent}, we see that $\V( s_{\alpha})$ has the following presentation:
\begin{center}
\begin{tikzpicture}
  \matrix (m) [row sep=1em,column sep=1em,minimum width=2em]
     {
     \node(a){$\O_{X}^{n}$};&\node(){}; &\node(){}; &\node(){}; &\node(){}; &\node(f){$\I\otimes\L$}; &\node(d){$\O_{\V(s_{\alpha})}$};& \node(e){$0$}; \\};

  \path[-stealth]
    (a) edge node[above]{{\small $s_{\alpha}=(\underset{i=0}{\overset{n}{\sum}}a_{i1}\phi_i\;\ldots\; \underset{i=0}{\overset{n}{\sum}}a_{in}\phi_i)$}}  (f)
    (f) edge (d)
    (d) edge (e);
\end{tikzpicture}
\end{center}
where $(a_{ij})_{0\leq i\leq n,1\leq j\leq n}$ is an $(n+1)\times n$ generic matrix with entries in the field $\k$. Since by definition, $[\V(\cos(\E))]= [\V( s_{\alpha})]$ we have that  $\deg\Big([\V(\cos(\E))]\Big)=\lgth(\O_{\V( s_{\alpha})})=\delta^n-\tau(Z)$ by definition of $\tau(Z)$.

The equalities \ref{MuTauT} and \ref{MuTauDiff} follow in the same way from the definition of $\mu(Z)$ and $\tau(Z)$ and from the decomposition of $\PI$ as the union of $\tX$ and $\T_Z$.
\end{proof}

We now explain how to practically compute the number $\mu(Z)$.

\begin{prop}\label{generalCasualMilnor} Let $(a_{ij})_{0\leq i\leq n,1\leq j\leq n}$ be an $(n+1)\times n$ generic matrix with entries in the field $\k$. Then, denoting by $(\underset{i=0}{\overset{n}{\sum}}a_{i1}\phi_i,\ldots, \underset{i=0}{\overset{n}{\sum}}a_{in}\phi_i)_z$ the localisation at $z$, we have:
\[\mu(Z,z)=\lgth(\O_{X,z}/(\underset{i=0}{\overset{n}{\sum}}a_{i1}\phi_i,\ldots, \underset{i=0}{\overset{n}{\sum}}a_{in}\phi_i)_z).\]
\end{prop}

\begin{proof}
Recall that $d_t(\Phi)$ can be computed in the following way. A generic point $y\in\projnk$ is the intersection of $n$ general hyperplanes $L_j:\underset{i=0}{\overset{n}{\sum}}a_{ij}x_i=0$, that is, the data of an $(n+1)\times n$ generic matric $N$ with entry in the field $\k$. The preimage of $y$ by $\Phi$ is contained in the scheme $\bbF'=\V(\underset{j=0}{\overset{n}{\sum}}a_{1j}\phi_j,\ldots, \underset{j=0}{\overset{n}{\sum}}a_{nj}\phi_j)$. Hence, to compute the topological degree of $\Phi$, it remains to remove the points of $\bbF'$ in the base locus. But since the formation of the symmetric algebra commutes with base change, we can suppose that $Z$ consists of a single point $z$.

So \[d_t(\Phi)=\lgth(\bbF)=\delta^n-\lgth(\O_{X,z}/(\underset{i=0}{\overset{n}{\sum}}a_{i1}\phi_i,\ldots, \underset{i=0}{\overset{n}{\sum}}a_{in}\phi_i)_z),\] which implies that $\lgth(\O_{X,z}/(\underset{i=0}{\overset{n}{\sum}}a_{i1}\phi_i,\ldots, \underset{i=0}{\overset{n}{\sum}}a_{in}\phi_i)_z)=\tau(Z,z)+\deg(T_z)=\mu(Z,z).$
\end{proof}

\begin{rem} From a more computational point of view, letting \[\bbF'=\V(\underset{j=0}{\overset{n}{\sum}}a_{1j}\phi_j,\ldots, \underset{j=0}{\overset{n}{\sum}}a_{nj}\phi_j)\] as in the proof of \ref{generalCasualMilnor}, the preimage of $y$ is equal to the scheme \[\bbF=\V(\underset{z\in Z}{\cap}[(\underset{i=0}{\overset{n}{\sum}}a_{i1}\phi_i,\ldots, \underset{i=0}{\overset{n}{\sum}}a_{in}\phi_i):(\underset{i=0}{\overset{n}{\sum}}a_{i1}\phi_i,\ldots, \underset{i=0}{\overset{n}{\sum}}a_{in}\phi_i)_z])\] 
where, given two ideals $J$ and $J'$ of a ring $R$, we let $[J:J']$ be the \emph{ideal quotient} (see \cite[page 15]{eisenbud1995algebra}).

\end{rem}

\subsection{The polar case}\label{polarCase}

In the polar case, $X$ is the projective space $\Pnk$ over $\k$.
\begin{defi}
Let $F=\lbrace f=0\rbrace$ be a hypersurface in $\projnk$ where $f$ is a homogeneous polynomial of degree $d$ in $\coordnk$. Let $f_i=\frac{\partial f}{\partial x_i}$ and $\I=(f_0,\ldots, f_n)$ be the ideal sheaf in $\O_{\P^n}$ generated by the partial derivatives of $f$, called the \emph{jacobian ideal} of $f$. Recall that we call the map $\Phi_f$ associated to $\I$ the \emph{polar map}.

The topological degree of $\Phi_f$ is called the \emph{polar degree of} $F$.
\end{defi}

In order to use the Euler identity, we suppose in the sequel that the characteristic of the base field does not divide the degree of the polynomial $f$ defining the hypersurface $F$. We also always assume that the jacobian ideal $\I$ of $F$ is zero-dimensional.

We recall the classical definition of Milnor and Tjurina numbers. 

\begin{defi}\label{UsMuTau} Let $z\in Z=\V(\I)$ and via a change of coordinates, suppose that $z=(1:0:\ldots:0)$. Set $g_{\flat}\in \k[x_1,\ldots,x_n]$, the usual deshomogeneisation of a homogeneous polynomial $g\in \coordnk$ in the chart $\lbrace x_0\neq 0\rbrace $.

The \emph{local Tjurina number} at $z$, denoted by $\tau_f(Z,z)$ is defined as \[ \tau_f(Z,z)=\lgth\Bigl(\O_{\k^n,z}/(f_{\flat},(f_{\flat})_1,\ldots,(f_{\flat})_n)\Bigr) \hspace{0.5cm}\text{where }(f_{\flat})_i=\frac{\partial f_{\flat}}{\partial x_i}.\]

The \emph{local Milnor number} at $z$, denoted by $\mu_f(Z,z)$, is defined as \[ \mu_f(Z,z)=\lgth\Bigl(\O_{\k^n,z}/((f_{\flat})_1,\ldots,(f_{\flat})_n)\Bigr) \hspace{0.5cm}\text{where }(f_{\flat})_i=\frac{\partial f_{\flat}}{\partial x_i}.\]

The \emph{global Tjurina number} of $F$, denoted by $\tau_f(Z)$ (resp. \emph{global Milnor number} of $F$, denoted by $\mu_f(Z)$) is the sum $\sum \tau_f(Z,z)$ (resp. $\sum \mu_f(Z,z)$) over all $z\in Z$.
\end{defi}

We explain now how the numbers $\mu(Z)$ and $\tau(Z)$ defined in \Cref{defMilnor} coincide with the usual definitions of Milnor and Tjurina number given in \Cref{UsMuTau}.

\begin{prop}\label{propEgMuTau} Let $F=\lbrace f=0\rbrace$ be a reduced hypersurface in $\projnk$ where $f$ is a homogeneous polynomial in $\coordnk$ of degree $d$. Let $z\in Z=\V(\I)$ then:
\[ \tau(Z,z)=\tau_f(Z,z)\hspace{0.5cm}\text{and}\hspace{0.5cm}\mu(Z,z)=\mu_f(Z,z). \]
\end{prop}

\begin{proof}
Via a change of coordinates, we can suppose $z=(1:0:\ldots:0)$. The deshomogenisation of the Euler identity in the chart $\lbrace x_0\neq 0\rbrace$ is:

\[ df_{\flat}= (f_0)_{\flat}+\underset{i= 1}{\overset{n}{\sum}} x_i(f_i)_{\flat}\]
and $(f_i)_{\flat}=(f_{\flat})_i$ for $1\leq i\leq n$. The equality 
\begin{center}
$((f_0)_{\flat},\ldots,(f_n)_{\flat})=(f_{\flat}, (f_{\flat})_1,\ldots,(f_{\flat})_n)$
\end{center}
implies that $\tau(Z,z)=\tau_f(Z,z)$.

For the Milnor number, we let $A=(a_{ij})_{0\leq i\leq n,1\leq j\leq n}$ a generic $(n+1)\times n$ matrix with entries in the field $\k$. By \Cref{generalCasualMilnor}, \[\mu(Z,z)=\lgth\Bigl(\O_{\projnk,z}/(\underset{i=0}{\overset{n}{\sum}}a_{i1}f_i,\ldots, \underset{i=0}{\overset{n}{\sum}}a_{in}f_i)_z\Bigr).\]
By localisation at $z$, we have that $\mu(Z,z) =\lgth(\O_{M_A})$ where $M_A$ is defined as the cokernel of the following composition map:
\begin{center}
\begin{tikzpicture}
  \matrix (m) [row sep=0em,column sep=2em,minimum width=2em]
  {
     &\node(c){$\O_{z}^{n}$}; & \node(a){$\O_{z}^{n+1}$}; &\node(){}; &\node(f){$\O_z$}; &\node(d){$\O_{M_A}$};& \node(e){$0$}; \\};

  \path[-stealth]
    (c) edge node[above]{$A$} (a)
   (a) edge node[above]{{\footnotesize $ (f_0 \hspace{0.1 cm}\ldots \hspace{0.1 cm} f_n)_z $}} (f)
    (f) edge (d)
    (d) edge (e);
\end{tikzpicture}
\end{center}
whereas $\mu_f(Z,z)=\lgth(\O_{M})$ where $M$ is defined as the cokernel of the following composition map:
\begin{center}
\begin{tikzpicture}
  \matrix (m) [row sep=0em,column sep=2.3em,minimum width=2em]
  {
     &\node(c){$\O_{z}^{n}$}; &\node(){}; &\node(a){$\O_{z}^{n+1}$}; &\node(){}; &\node(f){$\O_z$}; &\node(d){$\O_{M}$};& \node(e){$0$.}; \\};

  \path[-stealth]
    (c) edge node[above]{}(a)
   (a) edge node[above]{{\footnotesize $ (f_0 \hspace{0.1 cm}\ldots \hspace{0.1 cm} f_n)_z $}} (f)
    (f) edge (d)
    (d) edge (e);

\begin{tiny}
    \matrix (m) [row sep=0cm,%
			 column sep=0cm,%
			 minimum width=2em,%
			 matrix of math nodes,%
             left delimiter  = (,%
             right delimiter = )] at (-2.9,1)
{%
  \node(a11){0}; &\node(a12){}; &\node(a13){}; &\node(a14){0};  \\
  \node(a21){1}; &\node(a22){}; &\node(a23){}; &\node(a24){};  \\
  \node(a31){0}; &\node(a32){}; &\node(a33){}; &\node(a34){};  \\
  \node(a41){}; &\node(a42){}; &\node(a43){}; &\node(a44){0};  \\
  \node(a51){0}; &\node(a52){}; &\node(a53){0}; &\node(a54){1};  \\
};

\draw[loosely dotted] (a11)-- (a44);
\draw[loosely dotted] (a11)-- (a14);
\draw[loosely dotted] (a14)-- (a44);

\draw[loosely dotted] (a21)-- (a54);
\draw[loosely dotted] (a31)-- (a53);
\draw[loosely dotted] (a31)-- (a51);
\draw[loosely dotted] (a51)-- (a53);
\end{tiny}
\end{tikzpicture}
\end{center}
But, since $\rk(A)=n$, we have  $\lgth(\O_{M_A})=\lgth(\O_M)$.

\end{proof}

In the case when $\tau(Z,z)=\mu(Z,z)$ for a point $z\in Z$, $Z$ is also called quasi-homogeneous at $z$ in \cite{saito1980theory}. As an application of the previous proposition, we recover a result originally proved over the field $\mathbb{C}$ in \cite{dimcapap2003hypersurfacecomplements}.

\begin{prop}\label{propDimPap}
Let $F=\lbrace f=0\rbrace \subset \projnk$ be a reduced hypersurface of degree $d$ over an algebraically closed field $\k$. Let $\Phi_f=(f_0:\ldots:f_n)$ be the polar map of $f$ and assume that $\V(f_0,\ldots,f_n)$ is finite.

Then \[d_t(\Phi_f)=(d-1)^n-\mu_f(Z).\]
\end{prop}

\begin{proof}
Since $f$ has degree $d$ and $\V(f_0,\ldots,f_n)$ is finite, \Cref{propDimPap} follows from \Cref{propEgMuTau} and \Cref{CosectionLongueur}~\ref{MuTauT} since the polynomials $f_i$ have degree $d-1$.
\end{proof}
\section{Examples and applications in the plane}\label{ExApplic}
In this section, $X$ is the projective plane $\projdk$. Letting $\L$ be a line bundle $\O_{\P^2}(\delta)$ for some $\delta\geq 1$, we consider the sections $\phi_0,\phi_1,\phi_2$ associated to the map $\Phi$ as homogeneous polynomials of degree $\delta$. We assume that the base ideal $\I=(\phi_0,\phi_1,\phi_2)$ has codimension $2$.

As above, $\mathcal{E}$ is defined as the kernel of the evaluation map as in the following exact sequence:
\begin{equation}\label{exSeq1Plan}\stepcounter{SExactes}\tag{E\theSExactes}
\begin{tikzcd}
0 \ar{r}& \mathcal{E} \ar{r}& \O_{\projdk}^{3} \ar[rr,"{(\phi_0\;\phi_1\;\phi_2)}"]& &\I(\delta) \ar{r}& 0.
\end{tikzcd}
\end{equation}

Since $\E$ is reflexive of rank $2$, it is locally free \cite{hartshorne1980stable}.

For $i\in\lbrace 1,2\rbrace$, we denote by $c_i(\E)$ the first and second Chern classes of $\E$. The class $\cos(\E)$ of a generic cosection of $\E$ is equal to the second Chern class $c_2(\E^\vee)$ of $\E^\vee$ and $c_2(\E^\vee)=c_2(\E)$. 

From now on, we identify Chern classes with their degree in $\mathbb{Z}$.

\subsection{Free and nearly free sheaves of relations}\label{SubFree}

For this subsection, $\O$ stands for $\O_{\projdk}$.

\begin{defi}
A vector bundle $\F$ of rank $2$ over $\projdk$ is said to be \emph{free} of exponents $(d_1,d_2)$ if there exists $(d_1,d_2)\in\mathbb{N}^{*2}$ such that $\mathcal{F}\simeq \mathcal{O}(-d_1)\oplus \mathcal{O}(-d_2)$.

It is said to be \emph{nearly free} of exponents $(d_1,d_2)$ if it has a graded free resolution of the form:
\begin{center}
\begin{tikzpicture}
  \matrix (m) [row sep=3em,column sep=2em,minimum width=2em]
  {
    \node(a){$0$}; &\node(b){$\O(-d_2-1)$}; &\node(c){$\O(-d_1)\oplus \O(-d_2)^2$}; &\node(u){}; & \node(d){$\F$}; & \node(e){$0$.}; \\};
  \path[-stealth]
    (a) edge (b)
    (b) edge (c)
    (c) edge node [above] {$(\phi_0\;\phi_1\;\phi_2) $} (d)
    (d) edge (e);
\end{tikzpicture}
\end{center}
\end{defi}

\begin{defi}[\cite{dimcaSti2015NearlyFreeDivRatCuspPlaneCurves}]\label{remdeffreecurves}
In the case where $\phi_0=f_0$, $\phi_1=f_1$, $\phi_2=f_2$ are the partial derivatives of a given squarefree polynomial $f\in \coorddk $, the curve $F=\lbrace f=0 \rbrace $ is called free (resp. nearly-free) if $\E$ in \eqref{exSeq1Plan} is free (resp. nearly-free).
\end{defi}

If $\phi_0,\phi_1,\phi_2$ are the partial derivatives of a squarefree homogeneous polynomial $f$, a result of A.A. du Plessis and C.T.C.Wall in \cite{duplessisWall19991higlysingular} identifies in particular curves $F=\lbrace f=0\rbrace$ of a given degree $d$ with maximal possible global Tjurina number. These are the free curves of exponents $(1,d-2)$. By \Cref{theorMuTau}, the maximality of this Tjurina number  is equivalent to the minimality of the second Chern class $c_2(\E)$ of the vector bundle $\E$ associated to $F$. The following theorem is thus a generalisation of the result of du Plessis-Wall. We emphasize that in this case $c_1(\E)$ is negative and $c_2(\E)$ is positive.

\begin{thm}\label{propFibre}
Let $\E$ as in \eqref{exSeq1Plan}, then:
\begin{enumerate}
\item \label{propFibre1} $-c_1(\E)\leq c_2(\E)+1$ and equality holds if and only if $\E$ is free of exponents $\Big(1,c_2(\E)\Big)$,
\item \label{propFibre2} in the case $c_1(\E)\leq -5$, $\E$ is nearly free of exponents $\Big(1,c_2(\E)\Big)$ if and only if $-c_1(\E)=c_2(\E)$.
\end{enumerate}
\end{thm}

\begin{proof}

We denote by $c_1$ and $c_2$ respectively the first Chern class $c_1(\E)$ and the second Chern class $c_2(\E)$ of $\E$. We let $c=-1-c_1\geq 0$ and \[m=\min\lbrace t\in\mathbb{Z}\text{ , }\textnormal{H}^0\Bigl(\projdk,(\E(t)\Bigr)\neq 0 \rbrace.\]

\begin{itemize}
\item[\ref{propFibre1}]Assume that $c_2\leq c$. We are going to show that the only possibility is that $c_2=c$ and $m=1$. First, $m>0$ since otherwise, if $0\neq s\in \tnH^0(\projdk,\E)$ we would have had $\E\simeq \O\oplus\O(-1-c)$ which contradicts the fact that $c_2>0$.

Now, let $s\in \tnH^0\Bigl(\projdk,\E(m)\Bigr)$ be a non zero section. Since $m$ is minimal, we have the following exact sequence:
\begin{equation}\label{suiteExSection}\stepcounter{SExactes}\tag{E\theSExactes}
\begin{tikzcd}[row sep=1em,column sep=2em,minimum width=2em]
0 \ar{r}& \O(-m) \ar{r}& \E \ar{r}& \I_L(m-1-c) \ar{r}& 0
\end{tikzcd}
\end{equation}
where $L\subset\projdk$ is a $0$-dimensional subscheme of length $l\geq 0$. It is a computation to show that $l=c_2-m(c+1-m)\geq 0$, and since $c_2\leq c$, we have 
\begin{equation}\label{ineqLong}
c(1-m)\geq m(1-m).
\end{equation}

So

\begin{itemize}
\item[(i)] if $m=1$, then $l=0$, i.e.\ $\I_L(m-1-c)=\O(m-1-c)$ and the sequence \eqref{suiteExSection} splits showing that $\E\simeq \O(-1)\oplus\O(-c)$,
\item[(ii)] if $m\geq 2$ then $m\geq c$.
\end{itemize} 

Now, assume by contradiction that $m\geq 2$. First, it follows from the Riemann-Roch formula that:
\[\chi\Big(\E(1)\Big)=\dfrac{8-2c_2-3c+c^2}{2}\geq \dfrac{8-5c+c^2}{2}.\]
Hence  $\chi\Big(\E(1)\Big)>0$ for all $c$. On the other hand, since $m\geq 2$, by (ii), $m\geq c$ and we have $\textnormal{H}^0\Bigl(\projdk,\E(1)\Bigr)=\textnormal{H}^2\Bigl(\projdk,\E(1)\Bigr)=0$ where the second vanishing follows the first, using from Serre-duality $\textnormal{H}^2\Bigl(\projdk,\E(1)\Bigr)\simeq\textnormal{H}^0\Bigl(\projdk,\E(c-3)\Bigr)^\vee$. These two vanishings contradict the fact that $\chi\Big(\E(1)\Big)>0$. Summing up, if $c_2\leq c$, the only possibility is $c_2=c$ and then $\E\simeq\O(-1)\oplus\O(-c)$ which completes the proof of \ref{propFibre1}.

\item[\ref{propFibre2}] It is a computation to show that if $\E$ is nearly free of exponents $(1,c_2)$,  then $c_2=c+1=-c_1$. Now, we assume that $c_2=c+1$ and that $c\geq 4$ and we show that $\E$ is nearly-free of exponents $(1,c_2)$.
From the inequality \eqref{ineqLong}, we obtain:

\begin{itemize}
\item[(i)] $m\geq 3$ implies $m\geq c$ and thus $\textnormal{H}^0\Big(\projdk,\E(c-1)\Big)=\textnormal{H}^0\Big(\projdk,\E(1)\Big)=0$. Then, the Riemann-Roch formula implies that
\begin{center}
$\chi\Big(\E(1)\Big)=\frac{(c-2)(c-3)}{2},$
\end{center} hence $\chi\Big(\E(1)\Big)>0$ for $c\geq 4$. As above this leads to a contradiction and so this case does not occur.

\item[(ii)] $m=2$ implies $c\leq 3$, a case excluded by the assumption $c\geq 4$.
\item[(iii)] $m=1$ implies  that $l=1$ where $l$ is the length of the scheme $L$ as in the exact sequence \eqref{suiteExSection}. Now, using the resolution of a point $p$ in $\projdk$, we get the following diagram:
\begin{footnotesize}
\begin{center}
\begin{tikzpicture}
  \matrix (m) [row sep=1.5em,column sep=2em,minimum width=2em]
  {
  	 \node(a){};&\node(a){};&\node(a){};&\node(u){$0$};&\node(a){};\\
     \node(b){$0$ };&\node(e){$ \mathcal{O}(-1)$}; & \node(f){ $\E$}; & \node(g){$\mathcal{I}_p(-c)$}; &\node(h){$0$}; \\
     \node(x){};&\node(x){};&\node(x){};&\node(v){$\O(-1-c)^2$};&\node(x){};\\
     \node(y){};&\node(y){};&\node(y){};&\node(w){$\O(-2-c)$};&\node(y){};\\
     \node(z){};&\node(z){};&\node(z){};&\node(t){$0$};&\node(z){};\\};
  \path[-stealth]
    (b) edge (e)
    (e) edge node[auto]{$\alpha$} (f)
    (f) edge (g)
    (g) edge (h)
    (w) edge (v)
    (t) edge (w)
    (v) edge (g)
    (g) edge (u)
    (v) edge[dash pattern=on 2pt off 2pt] node[auto]{$\beta$} (f);
\end{tikzpicture}
\end{center}
\end{footnotesize}
where the existence of $\beta$ is provided by the vanishing of $\Ext^1(\mathcal{O}(-1-c)^2,\mathcal{O}(-1))$ (see also \cite{MarValles2017NFCurvesBundle} for more details in this direction). Since $\E$ is locally free of rank $2$, the complex \eqref{exSeqDefNF} provides a locally free resolution of $\E$ showing that $\E$ is nearly-free of exponent $(1,-1-c)$ that is $\E$ has the resolution:
\begin{equation}\label{exSeqDefNF}\stepcounter{SExactes}\tag{E\theSExactes}
\begin{tikzcd}[row sep=1em,column sep=1.5em,minimum width=2em]
0 \ar{r}& \O(-c-2) \ar{r}& \O(-1)\oplus\O(-c-1)^2 \ar{r}& \E \ar{r}& 0.
\end{tikzcd}
\end{equation}
\end{itemize}
\end{itemize}

\end{proof}

As an application we recover \cite[Corollary 2.6]{dorHassSim2012polar} but with a different proof. Recall that $\I$ is said to be of \emph{linear type} if $\PI=\tX$, see the beginning of \Cref{subSecReesSym}.
\begin{cor}\label{theorCentral}
If $\I=(\phi_0,\phi_1,\phi_2)$ is of linear type then the associated map $\Phi$ is birational only if $\delta\leq 2$.
\end{cor}

\begin{proof}
Indeed, letting $\E$ be as in \eqref{exSeq1Plan}, we have that $c_2(\E)=d_t(\Phi)$. But $c_1=-\delta$ so the only possibility to have $d_t(\Phi)=1$ is that $\delta\leq 2$.
\end{proof}

\subsection{Homaloidal curves}\label{SubHomalo}

Now, let $\Phi_f$ be the polar map from $\projdk$ to $\projdk$ associated to a reduced plane curve $F=\lbrace f=0\rbrace\subset\projdk$ as in \Cref{polarCase}. In this case, \Cref{theorCentral} says that, if the singular locus of the curve $F$ is a local complete intersection, the curve is homaloidal only if $d\leq 3$. This extends partially to any algebraically closed field the result in \cite{dolgachev2000polar}.

Now, recall that for any singular point $z$ of the curve $F$, the \textit{conductor invariant} $\delta_z$ is defined as the length of the quotient module $\tilde{\O}_{F,z}/\O_{F,z}$ where $ \tilde{\O}_{F,z}$ is the normalisation of the local ring $\O_{F,z}$. The number of local branches of $F$ at $z$ is denoted by $r_z$. 

The combination of the Jung-Milnor formula over $\mathbb{C}$: \[\tau(Z,z)\leq\mu(Z,z)=2\delta_z-r_z+1\] and the formula for the arithmetic genus of a plane curve \cite[Part 3, Lemma 3 and Lemma 4]{dolgachev2000polar} gives the following relation:
\begin{align*}
\sum(r_z-1)&\leq 2\underset{i=1}{\overset{h}{\sum}}(1-g_i)+c_2(\E)-(d+1).
\end{align*}
Now, if $F$ verifies $c_2(\E) = d-2$, this inequality becomes:
\begin{align*}
\sum(r_z-1)&\leq 2\underset{i=1}{\overset{h}{\sum}}(1-g_i)-3
\end{align*}
where $h$ is the number of irreducible components $F_i$ of $X$ and $g_i$ is the genus of the normalization of $F_i$.
But $r_z\geq 1$, so $h>1$.
A direct consequence is the following proposition which elucidates a part of the structure of the curves with the smallest possible $c_2(\E)$ identified in \Cref{propFibre}.

\begin{prop}\label{propIrreductibles} Suppose that the field $\k$ is $\mathbb{C}$. Let $F=\lbrace f=0\rbrace\subset\projdC$ be a reduced plane curve of degree $d$ and let $\I$ be the ideal sheaf generated by the partial derivatives of $f$ and $\E$ be as in \eqref{exSeq1Plan}.

If $d=c_2(\E)+2$ then $F$ is reducible.
\end{prop}

This gives in particular another proof to the result in \cite[Th. 2.5 (iv)]{dimcasti2015freedivratcuspplanecurves}.

\subsubsection{In characteristic 3, a homaloidal curve of degree $5$}\label{exChar3Homalo}

In \cite{dolgachev2000polar}, the classification of complex homaloidal plane curves relies on the analysis of the Jung-Milnor's formula. In \cite{BouGreMar2012InvarHyperSingPosChar}, the authors showed that the Jung-Milnor formula applies over a field of characteristic $p>0$ provided that $F$ has no \emph{wild vanishing cycle} (see \cite{BouGreMar2012InvarHyperSingPosChar} for a definition) and in \cite{Duc2016InvPosChar}, a sufficient condition for an irreducible cuvre $F$ to have no wild vanishing cycle is to have degree $d$ such that $d(d-1)<p$. A rough idea is that for every $d$ such that the characteristic $p$ is way greater, the classification of homaloidal curves of degree $d$ remains the same. The following proposition shows that the classification differs when the degree is big enough compared to the characteristic and answers \Cref{pbClassif}.

\begin{prop}\label{exChar3}
The curve $F=\V\Big((x_1^2+x_0x_2)x_0(x_1^2+x_0x_2+x_0^2)\Big)$ is homaloidal if and only if the base field $\k$ has characteristic $3$, in which case the inverse of the polar map is
\begin{center}$\Psi=(-x_1^2x_2^2-x_0x_2^3-x_2^4:x_1^3x_2+x_0x_1x_2^2+x_1x_2^3:x_1^4+x_0x_1^2x_2+x_0x_2^3)$
\end{center}
\end{prop}

\begin{proof}
The curve $F$ is defined over $\mathbb{Z}$ hence over $\mathbb{F}_p$ for every $p$. The resolution of the jacobian ideal $\I$ over $\mathbb{Z}$ is as follows:

\begin{equation}\label{resIZ}\stepcounter{SExactes}\tag{R\theSExactes}
\begin{tikzcd}[ampersand replacement=\&,row sep=3em,column sep=0.95em,minimum width=2em]
0 \ar{r}\& \O(-1)\oplus\O(-3) \arrow[rrrrrrrrrrrrrr, "{\begin{pmatrix}
    0 &  2x_0^3+4x_0x_1^2+4x_0^2x_2 \\ x_0 & -x_1^3 \\ -2x_1 & -6x_0x_1^2-8x_0^2x_2-8x_1^2x_2-6x_0x_2^2
    \end{pmatrix}}"] \&  \& \& \& \& \& \& \&  \& \& \& \& \& \& \O^3 \ar{r}\& \I(4) \ar{r}\& 0
\end{tikzcd}
\end{equation} where we denote $\O$ for the sheaf $\O_{\projdk}$.

We observe that for every prime $p\neq 2$ the reduction modulo $p$ of \eqref{resIZ} provides a resolution of $\I_p=\I\otimes_{\mathbb{Z}} \mathbb{F}_p$. In every characteristic $p\geq3$, $\Fitt_2\I_p=(x_0,x_1)$ so $\I_p$ is not a complete intersection and $\P(\I_p)$ has a torsion component above the point $z=(0:0:1)\in\projdk$. 

Moreover, in characteristic other than $2$, the resolution of $\PI_p=\P(\I_p)$ embedded in $\projnk\times\projnk$ is as follow:
\begin{equation*}
\begin{tikzcd}[row sep=3em,column sep=2em,minimum width=2em]
0 \ar{r}& \O(-4,-2) \ar{r}& \O(-1,-1)\oplus\O(-3,-1) \ar{r}& \I_{\PI_p} \ar{r}& 0
\end{tikzcd}
\end{equation*}
where $\O$ stands for $\O_{\Pnk\times\Pnk}$ and we wrote to the right the shift in the variables of the second factor of the product $\Pnk\times\Pnk$. From this resolution, we can compute that $\tau(Z,z)=13$ in every characteristic other that $2$.

In characteristic $3$, $\I_3$ has the following resolution:
\begin{center}
\begin{tikzpicture}
  \matrix (m) [row sep=3em,column sep=1.5em,minimum width=2em]
  {
     \node(a){$0$}; &\node(c){$\mathcal{O}(-1)\oplus\O(-3)$};&\node(u){};&\node(u){}; & \node(d){$\mathcal{O}^3$};& \node(e){$\I_3(4)$}; & \node(f){$0$.}; \\};
  \path[-stealth]
    (a) edge (c)
    (c) edge node[above]{\begin{scriptsize}
    $\begin{pmatrix}
    0 & x_0^3-x_0x_1^2-x_0^2x_2 \\ x_0 & x_1^3 \\ x_1 & -x_0^2x_2-x_1^2x_2
    \end{pmatrix}$
    \end{scriptsize}} (d)
    (d) edge (e)
    (e) edge (f);
\end{tikzpicture}
\end{center}

The difference in characteristic $3$ comes from the multiplicity of the torsion component in $\PI_3$. Indeed, the torsion component $\T_Z$ has the following resolution over $\mathbb{Z}$:

\begin{center}
\begin{tikzpicture}
  \matrix (m) [row sep=5em,column sep=2em,minimum width=2em]
  {
     \node(a){$0$}; &\node(c){$\O(-2,0)$}; & \node(d){$\O(-1,0)^2$}; & \node(e){$\I_{\T_Z}$}; & \node(f){$0$}; \\};
  \path[-stealth]
    (a) edge (c)
    (c) edge (d)
    (d) edge (e)
    (e) edge (f);
\end{tikzpicture}
\end{center}
whereas in characteristic $3$, it has resolution:

\begin{center}
\begin{tikzpicture}
  \matrix (m) [row sep=5em,column sep=2em,minimum width=2em]
  {
     \node(a){$0$}; & \node(b){$\O(-3,-1)$}; & \node(c){$\begin{matrix}
     \O(-3,0)^2 \\ \oplus \\ \O(-2,-1)^2\end{matrix}  $}; & \node(d){$\begin{matrix}\O(-2,0)^3 \\ \oplus \\ \O(-1,-1) \end{matrix}$}; & \node(e){$\I_{\T_{Z_3}}$}; & \node(f){$0$.}; \\};
  \path[-stealth]
    (a) edge (b)
    (b) edge (c)
    (c) edge (d)
    (d) edge (e)
    (e) edge (f);
\end{tikzpicture}
\end{center}

To sum up $\mu(Z,z) = 15$ and $d_t(\Phi)=1$ in characteristic $3$ or else  $\mu(Z,z) = 14$ and $d_t(\Phi)=2$ in other characteristic different from $2$ and $3$. In characteristic $3$, the polar map can be written \begin{center}$\Phi_f=(x_1^4+x_0^3x_2+x_0x_1^2x_2:-x_0^3x_1+x_0x_1^3+x_0^2x_1x_2,x_0^4-x_0^2x_1^2-x_0^3x_2).$ \end{center}
and it is a computation to check that $\Psi$ is the inverse of $\Phi_f$.

\end{proof}

\begin{rem}
What we did is to deepen the multiplicity of the torsion component by specializing the resolution of $\I$ over $\mathbb{Z}$ modulo a prime $p$ for which some monomials of the presentation matrix disappear (here $p=3$ works). We emphasize that in characteristic $3$, the torsion part $\T_Z$ is not equal scheme-theoretically to $\P^n_{\Fitt_n\I}$ whereas it is in greater characteristic. It is not clear if such an example is sporadic or not.
\end{rem}

\subsubsection{The reduction problem in positive characteristic}
The analysis of the presentation of the jacobian ideal gives also an easy way to construct examples of non reduced plane curves in positive characteristic where the topological degree is not preserved by reduction. It suffices to compute the presentation matrix of the jacobian ideal and adjust the characteristic of the field in order to modify the first syzygy matrix. 

The next proposition answers \Cref{pbRed}. We emphasize that, in the examples we consider, none of the exponents divide the characteristic of the field and that the characteristic $101$ does not play a particular role in comparison to other primes.

\begin{prop} Let $k$ be an algebraically closed field of characteristic $101$.
\begin{enumerate}[\rm{\it(\roman*)}]

\item\label{redprob1} The curve $\V\Big(z(y^3+x^2z)\Big)$ has polar degree $2$ whereas $\V\Big(z^{50}(y^3+x^2z)^{51}\Big)$ has polar degree $1$.

\item\label{redprob2} The curve $\V\Big((y^3+x^2z)(y^2+xz)\Big)$ has polar degree $5$ whereas the curve $\V\Big((y^3+x^2z)^{31}(y^2+xz)^4\Big)$ has polar degree $3$.

\end{enumerate}

\end{prop}

\begin{proof} Both curves are defined over $\mathbb{Z}$ and as in the proof of \Cref{exChar3}, the idea is to take reduction modulo the prime $p=101$ of the resolution of their jacobian ideal over $\mathbb{Z}$ to get a resolution over $\mathbb{F}_p$. We give the complete argument for \Cref{redprob1}. \Cref{redprob2} is similar and left to the reader. As in the proof of \Cref{exChar3}, $\I_p$ stands for $\I\otimes_{\mathbb{Z}} \mathbb{F}_p$.

The jacobian ideal of $\V\Big(z(y^3+x^2z)\Big)=0$ has resolution 
\begin{center}
\begin{tikzpicture}
  \matrix (m) [row sep=3em,column sep=2em,minimum width=2em]
  {
     \node(a){$0$}; & \node(b){$\O(-1)\oplus \O(-2)$}; & \node(c){$\O^{3}$}; & \node(d){$\I_{101}(3)$}; & \node(e){$0$,}; \\};
  \path[-stealth]
    (a) edge (b)
    (b) edge (c)
    (c) edge node [above] {$\Phi_{red}$} (d)
    (d) edge (e);
\end{tikzpicture}
\end{center}
$\I_{\PI_{101}}$ has the following resolution:
\begin{small}
\begin{center}
\begin{tikzpicture}
  \matrix (m) [row sep=3em,column sep=2em,minimum width=2em]
  {
     \node(b){$0$}; & \node(c){$\O(-3,-2)$}; & \node(d){$\begin{matrix}\O(-1,-1) \\ \oplus \\ \O(-2,-1) \end{matrix}$}; & \node(e){$\mathcal{I}_{\PI_{101}}$}; & \node(f){$0$.}; \\};
  \path[-stealth]
    (b) edge (c)
    (c) edge (d)
    (d) edge (e)
    (e) edge (f);
\end{tikzpicture}
\end{center}
\end{small}
There is no torsion component above the point $z=(1:0:0)$ and so the corresponding polar map has topological degree $2$.

But the jacobian ideal of the curve $\V\Big(z^{50}(y^3+x^2z)^{51}\Big)$ has resolution
\begin{center}
\begin{tikzpicture}
  \matrix (m) [row sep=3em,column sep=2em,minimum width=2em]
  {
     \node(a){$0$}; & \node(b){$\O(-1)\oplus \O(-2)$}; & \node(c){$\O^{3}$}; & \node(d){$\I_{101}(3)$}; & \node(e){$0$}; \\};
  \path[-stealth]
    (a) edge (b)
    (b) edge (c)
    (c) edge node [above] {$\Phi $} (d)
    (d) edge (e);
\end{tikzpicture}
\end{center}
and $\I_{\PI_{101}}$ has the following resolution:

\begin{center}
\begin{tikzpicture}
  \matrix (m) [row sep=3em,column sep=2em,minimum width=2em]
  {
     \node(b){$0$}; & \node(c){$\O(-3,-2)$}; & \node(d){$\begin{matrix}\O(-1,-1) \\ \oplus  \\ \O(-2,-1)\end{matrix}$}; & \node(e){$\I_{\PI_{101}}$}; & \node(f){$0$.}; \\};
  \path[-stealth]
    (b) edge (c)
    (c) edge (d)
    (d) edge (e)
    (e) edge (f);
\end{tikzpicture}
\end{center}
There is a torsion component above the point $z=(1:0:0)$, what we can see from the resolution of $\tX$:

\begin{center}
\begin{tikzpicture}
  \matrix (m) [row sep=3em,column sep=2em,minimum width=2em]
  {
     \node(b){$0$}; & \node(c){$\O(-2,-2)^2$}; & \node(d){$\begin{matrix}\O(-1,-1) \\ \oplus \\ \O(-2,-1)\\\oplus \\ \O(-1,-2) \end{matrix}$}; & \node(e){$\I_{\tX}$}; & \node(f){$0$.}; \\};
  \path[-stealth]
    (b) edge (c)
    (c) edge (d)
    (d) edge (e)
    (e) edge (f);
\end{tikzpicture}
\end{center}

The polar map of the latter curve is given by \[(x:y:z)\mapsto (xz^2:-49y^2z:50y^3)\] and its inverse is $(x:y:z)\mapsto(-37xz^2:-3y^2z:y^3).$
\end{proof}

\bibliographystyle{alpha}

\end{document}